\documentclass[english,12pt,a4paper,reqno]{amsart}
\usepackage{amsmath,amsthm,amsfonts,amssymb,amscd}
\usepackage{url}
\usepackage[ansinew]{inputenc}
\usepackage{graphics}
\usepackage{graphicx}
\usepackage{caption}
\usepackage{subcaption}
\usepackage{a4wide}
\usepackage{hyperref}
\usepackage{xfrac}
\usepackage{array}
\usepackage[all]{xy}
\usepackage{epic}
\usepackage{latexsym}
\usepackage{enumerate}
\usepackage{appendix}
\usepackage{mathrsfs}
\def\dist{\ensuremath{\text{dist}}}

\def\const{\ensuremath{\text{const}}}

\newtheorem{theorem}{Theorem}[section]
\newtheorem{proposition}[theorem]{Proposition}
\newtheorem{lemma}[theorem]{Lemma}

\newtheorem{maintheorem}{Theorem}

\newcommand{\cmt}{\begin{maintheorem}}
\newcommand{\fmt}{\end{maintheorem}}

\newtheorem{maincorollary}[maintheorem]{Corollary}

\newcommand{\cmc}{\begin{maincorollary}}
\newcommand{\fmc}{\end{maincorollary}}

\theoremstyle{remark}
\newtheorem{definition}{\bf Definition}

\begin{document}

\title[Statistical stability of Geometric Lorenz attractors]{Statistical stability of Geometric Lorenz attractors}

\author{Jos\'e F. Alves}
\address{Jos\'e F. Alves\\ Departamento de Matem\'atica, Faculdade de Ci\^encias da Universidade do Porto\\
Rua do Campo Alegre 687, 4169-007 Porto, Portugal}
\email{jfalves@fc.up.pt} \urladdr{http://www.fc.up.pt/cmup/jfalves}
\author{Mohammad Soufi}
\address{Mohammad Soufi\\ Departamento de Matem\'atica, Faculdade de Ci\^encias da Universidade do Porto\\
Rua do Campo Alegre 687, 4169-007 Porto, Portugal}
\email{msoufin@gmail.com}


\thanks{The authors were partially supported by Funda\c c\~ao Calouste Gulbenkian, by CMUP, by the European Regional Development Fund through the Programme COMPETE and by  FCT under the projects PTDC/MAT/099493/2008 and PEst-C/MAT/UI0144/2011.}



\begin{abstract}
We consider the robust family of geometric Lorenz attractors. These attractors are chaotic, in the sense that they are transitive and have sensitive dependence on the initial conditions. Moreover, they support  SRB measures whose ergodic basins cover a full Lebesgue measure subset of points in the topological basin of attraction. Here we prove that the SRB measures depend continuously on the dynamics in the weak$^\ast$  topology.    
\end{abstract}

\maketitle

\tableofcontents

\section{Introduction}

The theory of Dynamical Systems initiated by Poincar\'e's work on the three-body problem of celestial mechanics and it studies processes which are evolving in time. The description of the processes is given in terms of flows when the time is continuous or iterations of maps when the time is discrete. An orbit is a time-order collection of states of the system starting from a specific state applying the flow or the map. The main goals of this theory are to describe the typical behavior of orbits as time goes to infinity, and to understand how this behavior changes when we perturb the system or to which extent it is stable. In this work we are concerned with the stability of a system.

Ergodic Theory deals with measure preserving processes in a measure space. One in particular tries to describe the average time spent by typical orbits in different regions of the phase space. According to Birkoff's Ergodic Theorem, such times are well defined for almost all points, with respect to any invariant probability measure. However, the notion of typical orbit is usually meant in the sense of volume (Lebesgue measure), which is not always an invariant measure. It is a fundamental open problem to understand under which conditions the behavior of typical (with respect to Lebesgue measure) orbits   is well defined from the statistical point of view. This problem can be precisely formulated by means of \emph{Sinai-Ruelle-Bowen (SRB) measures} which were introduced by Sinai for Anosov diffeomorphisms \cite{sinai} and later extended by Ruelle and Bowen for Axiom A diffeomorphisms and flows~\cite{ruelle1,ruelle2}.

\begin{definition}[SRB measures]
Let $\mu$ be an invariant Borel probability measure for a flow~$(X^t)_t$ on the Borel sets of a manifold $M$.  The \emph{basin} of $\mu$ is the set of points $x\in M$ such that
\begin{equation}\label{de.srb}
\lim\limits_{T\to+\infty}\frac 1T\int_0^T\varphi(X^t(x))~dt=\int\varphi~d\mu,\quad\mbox{for any continuous }\varphi:M\to\mathbb{R}.
\end{equation}
The measure $\mu$ is called an \emph{SRB measure} if its basin has positive Lebesgue measure. 
\end{definition}

The notions of basin and SRB measure can easily be extended to discrete time dynamical systems, simply by replacing the integral by a time series in~\eqref{de.srb}.

A fairly good description of the statistical behavior of orbits can be given by an SRB measure in the sense that, for a ``big'' (meaning positive volume) set of points, the time averages of a physical observable (a continuous function on the manifold) of the system is accomplished simply by integrating the observable with respect to SRB measure (space average).

In trying to capture the persistence of the statistical properties of a dynamical system, Alves and Viana \cite{alves4} proposed a notion, called \textit{statistical stability}, which expresses the continuous variation of SRB measures as a function of the dynamical system. This is a kind of stability in the sense that the outcome of evaluating continuous functions along orbits does not change much under small perturbations of the system. This is what we might be observed in computer experiments, where typically the picture obtained by plotting an orbit seems to be independent of the starting point and truncation errors. 

Next we introduce the notion of statistical stability for vector fields. 

\begin{definition}[Statistical stability]
Assume we have a family $\mathcal X$ of vector fields endowed with a topology, admitting a common trapping region $U$ on which each $X\in \mathcal X$ has a unique SRB measure $\mu_X$. We say that $\mathcal X$ is \emph{statistically stable} (in $U$) if  the map $\mathcal X\ni X \longmapsto\mu_X$ is continuous, where in the space of probability measures we consider the weak$^\ast$ topology.
\end{definition}

Our goal in this work is to prove the statistical stability of a family of vector fields associated to the Lorenz equations.

\subsection{Lorenz equations} Lorenz \cite{lorenz} studied numerically a vector field $X$ defined by the system of equations
$$\left\{\begin{array}{l}
\dot{x}=a(y-x),\\
\dot{y}=bx-y-xz,\\
\dot{z}=xy-cz,
\end{array}
\right.$$
for the parameters $a=10$, $b=28$ and $c=8/3$. The following properties are well known for this vector field: \begin{enumerate}
\item $X$ has a \textit{singularity} at origin which $DX(0)$ has real eigenvalues $$0<-\lambda_3\approx 2.6<\lambda_1\approx 11.83<-\lambda_2\approx 22.83;$$ 
\item  
there is an open set $U$, \textit{trapping region}, such that $X^t(\bar{U})\subseteq U$ for all $t>0$.  The maximal invariant set in $U$, $\Lambda=\cap_{t>0}X^t(U)$, is an attractor and the origin is the only singularity contained in $U$;
\item  
the divergence of $X$ is negative:
$$
\mbox{div} X=\partial\dot{x}/dx+\partial\dot{y}/dy+\partial\dot{z}/dz=-(a+1+c)<0.
$$
By Liouville's Formula, the flow of $X$ contracts volume. Thus, $\Lambda$ has zero volume.
\end{enumerate}

Lorenz found with his experimental computations that the flow is sensitive with respect to the initial conditions near the attractor, i.e. even a small initial error lead to enormous differences in the outcome. It was a challenging problem to give a rigorous mathematical proof for this experimental evidence. Tucker \cite{tucker} gave a computer assisted proof that the original Lorenz system indeed corresponds to a robustly transitive non-hyperbolic attractor containing a singularity. Moreover, he proved that the Lorenz equations define a dynamical system with the behavior of the geometric model introduced by Guckenheimer and Williams \cite{williams} that we describe next.

\subsection{Geometric model}\label{section2}
Here we briefly describe the geometric model of the Lorenz attractor; see e.g. \cite{araujo2} for more details. The model is given by a vector filed $X_0$ which is linear in a neighborhood of the origin. The real eigenvalues $\lambda_1, \lambda_2$ and $\lambda_3$ of $DX_0(0)$ with the eigenvectors along the coordinate axis 
satisfy
$ 
0<-\lambda_3<\lambda_1<-\lambda_2.
$ 
\begin{figure}[h]
\begin{center}
\vspace{-.3cm}
\includegraphics[height=12cm, angle=-90]{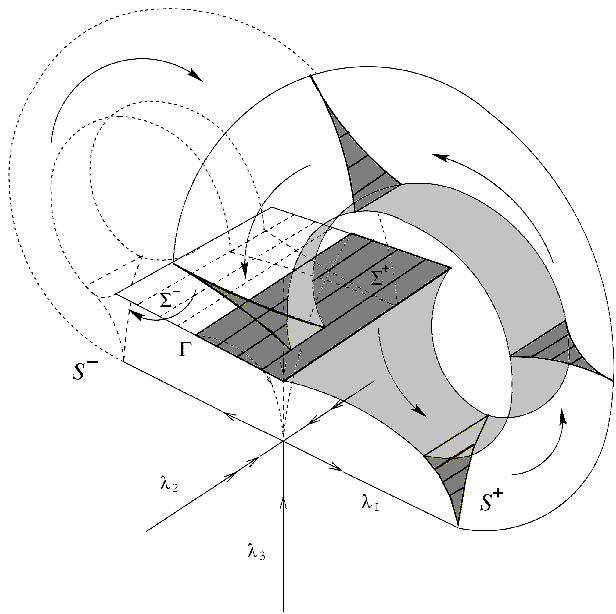}
\vspace{-5mm}
\caption{Geometric Lorenz flow}\label{fig:2}
\end{center}
\end{figure}
We consider the square  
given by $$\Sigma=\left\{(x,y,1): -\frac 12\leq x,y \leq\frac 12\right\},$$
and let $\Gamma$ be the intersection of $\Sigma$ with the local stable manifold of the singularity. 
The segment $\Gamma$ divides $\Sigma$ in to two parts
$$
\Sigma^+=\{(x,y,1)\in\Sigma: x>0\}\quad\text{and}\quad\Sigma^-=\{(x,y,1)\in\Sigma: x<0\},
$$
The images of $\Sigma^{\pm}$ by this map are curveline triangles $S^{\pm}$ without the vertexes $(\pm 1,0,0)$ and every line segment in $\mathcal{F}=\{x=\const\cap\Sigma\}$ except $\Gamma$ is mapped to a segment in $\{z=\const\cap S^{\pm}\}$.
The time $\tau$  which takes for each $(x,y,1)\in\Sigma\setminus\Gamma$ to reach $S^{\pm}$ is given by $\tau(x,y,1)=-\frac{1}{\lambda_1}\log|x|.$
Now we suppose that the flow takes the triangles back to the $\Sigma$ in a smooth way as it is shown in Figure \ref{fig:2}.
The resulting Poincar\'e map form $\Sigma\setminus\Gamma$ into $\Sigma$ has the form
\begin{equation}\label{eq.mapg}
P(x,y)=(f(x),g(x,y)),
\end{equation}
for some $f:I\setminus\{0\}\to I$ and $g:I\setminus\{0\}\times I\to I$, where $I=[-\frac 12, \frac 12]$.
The one-dimensional map $f$ is as described in Figure~\ref{fig:3} and satisfies:
\begin{enumerate}
\item $f$ has a discontinuity at $x=0$ with side limits $f(0^+)=-1/2$ and $f(0^-)=1/2$;
\item $f$ is differentiable on $I\setminus\{0\}$ and there is $c>1$ such that $f'(x)\ge c$ for all $x\in I\setminus\{0\}$;
\item the limit of $f'(x)$ is infinity as $x$ approaches $0^\pm$; 
\item $f''(x)>0$ for $x\in[-\frac 12,0)$ and $f''(x)<0$ for $x\in(0,\frac 12]$;
\item $f$ is transitive.
\end{enumerate}
\begin{figure}[h]
\centering
\includegraphics[width=6cm,angle=-90]{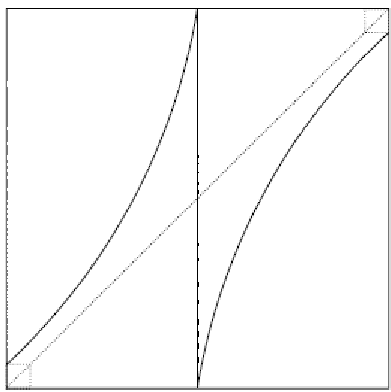}
\caption{Lorenz map}\label{fig:3}
\end{figure}

%

%

The map $g$ in equation \eqref{eq.mapg} is defined in such a way that  the stable foliation $\mathcal F$ is uniformly contracting: there exist constants $C' > 0$ and $0 < \rho < 1$ such that for any given leaf $\gamma$ of the foliation and $\xi_1, \xi_2\in\gamma$ and $n\geq 1$,
$$\dist(P^n(\xi_1),P^n(\xi_2))\leq C'\rho^n \dist(\xi_1,\xi_2).$$

\subsection{Statement of results}\label{subsection}

A crucial fact about the geometric Lorenz attractor is that it is robust, i.e. vector fields sufficiently close in the $C^1$ topology to the original one constructed as above  also have strange attractors. Indeed, there exist an open neighborhood $U$ in $\mathbb{R}^3$ containing the geometric Lorenz attractor $\Lambda$ and an open neighborhood $\mathcal{U}$ of $X_0$ in $C^1$ topology such that for all vector fields $X\in\mathcal{U}$, the maximal invariant set $\Lambda_X=\cap_{t\geq 0}X^t(U)$ is a transitive set which is invariant under the flow of $X$. This is a consequence of the persistence of an invariant contracting foliation $\mathcal{F}_X$ on the cross section $\Sigma$ for $X\in\mathcal{U}$; see \cite[Theorem 3.10]{araujo2}. 

Under some conditions on the eigenvalues of the singularity, for a $C^2$-close vector field $X$ to $X_0$, the leaves of $\mathcal{F}_X$ are $C^2$ close to those of $\mathcal{F}$ and it follows that $f_X$ is $C^2$ close to~$f$; see \cite{R81,R84}. Thus, there exits $o\in[-\frac 12,\frac 12]$ which play for $f_X$ the same role of 0 for $f$, and the properties of $f$ in Subsection \ref{section2} are still valid for $f_X$ on a subinterval $[-b,b]$, for some $0<b<\frac 12$ close to $\frac 12$. 

\begin{definition}\label{de.lorenzflows} We define  the family $\mathcal{X}$ of \emph{Geometric Lorenz vector fields }as a $C^2$ neighborhood of $X_0$  with the following properties:
\begin{enumerate}
\item for each $X\in \mathcal X$, the maximal forward invariant set $\Lambda_X$ inside $U$  is an attractor containing a hyperbolic singularity;
\item for each $X\in \mathcal X$, $\Sigma$ is a cross-section for the flow with a return time $\tau_X$ and a Poincar\'e map $P_X$;
\item for each $X\in\mathcal X$, the map $P_X$ admits a $C^2$ uniformly contracting  invariant foliation $\mathcal{F}_X$ on $\Sigma$ with projection along the leaves of $\mathcal F_X$ onto $I$ given by a map $\pi_X$;

\item for each $X\in\mathcal X$, the map $f_X$ on the quotient space $I$ by the leaves in $\mathcal F_X$ is a transitive $C^2$ piecewise  expanding  with two branches; moreover, there is $c>1$ such that $f'_X(x)\ge c$ except at the  discontinuity point $O_X$ and $\lim_{x\to O_X^\pm}f'_X(x)=+\infty$.

\item there is some constant $C>0$ such that for each $X\in\mathcal X$
\begin{equation}\label{eq.retime}
\tau_X(\xi)\le -C\log|\pi_X(\xi)-O_X|.
\end{equation}  
\end{enumerate}
\end{definition}

Observe that as the length of $I$ is equal to one, then $|\pi_X(\xi)-O_X|<1$ for all $X\in \mathcal X$. For a detailed exposition on the properties of geometric Lorenz flows see e.g. \cite[Section~2.3]{araujo2}; see also \cite[Equation (9)]{AV12} for the last property. The main goal of this work is to prove the following result.

\begin{maintheorem}\label{main}
Geometric Lorenz vector fields are statistically stable.
\end{maintheorem}

\section{Preliminaries}

Consider the family $\mathcal X$ of Geometric Lorenz vector fields as in Definition~\ref{de.lorenzflows}.
We assume that for each $X\in\mathcal X$ the derivative $f'_X$ is monotonic on each branch. On the other hand, $1/f'_X$ is bounded because $f'_X>1$. Therefore $1/f'_X$ is monotonic and bounded and hence is of bounded variation.  It follows from \cite[Corollary 3.4]{viana} that each  $f_X$ admits a unique ergodic invariant probability $\bar{\mu}_X$ which is absolutely continuous with respect to Lebesgue measure $\lambda$, whose density  $d\bar{\mu}_X/d\lambda$ is a bounded variation function and, in particular, it is bounded.  

We point out that statistical stability results for piecewise expanding maps have been obtained in \cite{keller}.
According to \cite[Corollary 14.]{keller} or \cite[Theorem 11.2.2.]{boyarsky}, the family $f_X$ with $X\in\mathcal X$ is in the conditions of the results by Keller. Moreover, the density $d\bar\mu_X/d\lambda$ can be obtained by means of the Lasota-Yorke inequality
whose constants can be taken the same for all Lorenz maps; see  \cite[Proposition 3.1.]{viana}.  Therefore the density functions $d\bar{\mu}_X/d\lambda$ are uniformly bounded \cite[Corollary 3.4]{viana}. Hence we have:

\begin{proposition}\label{pr.keller}
Each $f_X$ with $X\in\mathcal{X}$ is strongly statistical stable, i.e. $f_X\longmapsto d\bar{\mu}_X/d\lambda$ is continuous with respect to $L^1$-norm in the space of densities. Moreover, there exists $M>0$ such that for all $X\in\mathcal{X}$
we have $
d\bar{\mu}_X/d\lambda<M.
$
\end{proposition}

For any bounded function $\phi:\Sigma\to\mathbb{R}$, we define $\phi^\pm:I\to\mathbb{R}$ by
\begin{equation}\label{EQ.+-}
\phi^+(x)=\sup\limits_{\xi\in\pi_X^{-1}(x)}\phi(\xi)\quad\text{and}\quad \phi^-(x)=\inf\limits_{\xi\in\pi_X^{-1}(x)}\phi(\xi)\
\end{equation}
where $\pi_X:\Sigma\to I$ is the canonical projection by stable leaves. The next result is proved in \cite[Corollary 6.2]{araujo1}\label{cor1}. 

\begin{lemma}
There is a unique   $P_X$-invariant ergodic probability measure $\tilde{\mu}_X$ on $\Sigma$ such that for every continuous function $\phi:\Sigma\to\mathbb{R}$,
$$
\int\phi~d\tilde{\mu}_X=\lim\limits_{n\to\infty}\int{(\phi\circ P^n_X)^-d\bar{\mu}_X}= \lim\limits_{n\to\infty}\int{(\phi\circ P^n_X)^+d\bar{\mu}_X}.
$$
\end{lemma}

The measure  $\tilde{\mu}_X$ is an SRB measure for $P_X$ that we shall call the \emph{lift} of $\bar\mu_X$. Indeed, the uniform contraction of the stable leaves implies that the forward time averages of any pair $\xi_1, \xi_2$ of points on a same stable leaf for continuous function $\phi:\Sigma\to\mathbb{R}$ are equal 

$$
\lim\limits_{n\to\infty}\frac 1n\sum\limits_{j=0}^{n-1}\phi(P_X^j(\xi_1))=\lim\limits_{n\to\infty}\frac 1n\sum\limits_{j=0}^{n-1}\phi(P_X^j(\xi_2)).
$$ 
Hence the inverse image of the basin of $\bar{\mu}_X$ under $\pi_X$ is contained in the basin of $\tilde{\mu}_X$.  This shows that the basin of $\tilde{\mu}_X$ contains an entire strip of positive Lebesgue measure, because the basin of $\bar{\mu}_X$ is a subset of positive Lebesgue measure.

On the other hand, since the density $d\bar{\mu}_X/d\lambda$ is bounded, we conclude that the return time is integrable with respect to $\tilde{\mu}_X$. Then we can saturate this measure along the flow to obtain a unique SRB measure $\mu_X$ for the flow, supported on the attractor $\Lambda_X$, whose ergodic basin covers a full Lebesgue measure subset of points on the topological basin of attraction; see \cite[Section~7]{araujo1}.

\begin{proposition}
The flow of each $X\in\mathcal X$ has a unique SRB measure $\mu_X$ given for any continuous map $\varphi:U\to\mathbb{R}$ by
$$
\int\varphi~d\mu_X=\frac{1}{\tilde{\mu}_X(\tau_X)}\int\int_0^{\tau_X(\xi)}\varphi(X(\xi,t))dt d\tilde{\mu}_X(\xi),
$$
where $\tilde{\mu}_X(\tau_X)=\int\tau_X~d\tilde{\mu}_X$.
\end{proposition}

\section{Statistical stability for the Poincar\'e map} \label{se.poincare}

Here  we prove the statistical stability of the Poincar\'e maps on the cross-section $\Sigma$, i.e. the SRB measures $\tilde{\mu}_X$ depend continuously on the vector fields.
Let $(X_n)_{n\geq 1}$ be any sequence in $\mathcal X$ converging to $X\in\mathcal X$ in the $C^2$ topology. To shorten notations, we shall use subindex $n$ instead of $X_n$, for $n\geq 1$, and no subindex instead of $X$. 

Let $\phi:\Sigma\to\mathbb R$ be an arbitrary continuous function. 

\begin{lemma}\label{le.contleb} Given $m\ge 1$ and $\epsilon>0$, there is $n_0=n_0(m,\epsilon)$ such that for all $n\ge n_0$
 $$\int\left| (\phi\circ P_n^m)^+-(\phi\circ P^m)^+\right|~d\lambda<\epsilon.$$
\end{lemma}

\begin{proof} Given $m\ge 1$, we can write  $\displaystyle{\int|(\phi\circ P_n^m)^+-(\phi\circ P^m)^+|d\lambda}$ as the sum
\begin{equation}\label{sum.int}
\int_{B_n}|{(\phi\circ P_n^m)^+-(\phi\circ P^m)^+|d\lambda}
+\int_{B_n^c}|{(\phi\circ P_n^m)^+-(\phi\circ P^m)^+|d\lambda},
\end{equation}
where $B_n=\left\{\sum\limits_{i=0}^{m-1}\tau_n\circ P_n^i> N\right\}$ and $N=N(m)$ is some large number. Now,  by the last property in  the definition of the  geometric Lorenz flows and the fact that  the leaves of~$\mathcal F_n$ are nearly vertical lines,   there is some constant $C_1>0$ such that
 \begin{eqnarray*}
\lambda(B_n)&\le &C_1\sum_{i=0}^{m-1}\left|\{x\in I: -C\log|f_n^i(x)-O_n|>N\}\right|\\
&\le &
C_1\sum_{i=0}^{m-1} \left|f_{n}^{-i}\left(O_n-e^{-\frac{N}{C}},O_n+e^{-\frac{N}{C}}\right)\right|\\
&\le & C_1\sum_{i=0}^{m-1}\left(2/c\right)^i e^{-\frac{N}{C}},
\end{eqnarray*}
where $c>1$ is the uniform lower bound for the derivative.
As $\phi$ is bounded, the first integral in \eqref{sum.int} can be made arbitrarily small, provided $N$ is big enough.

We now estimate the second integral in  \eqref{sum.int}. Considering
$$
A_n=\left\{\xi: \left|\sum\limits_{i=0}^{m-1}\left(\tau_n\circ P_n^i\right)(\xi)-\sum\limits_{i=0}^{m-1}\left(\tau\circ P^i\right)(\xi)\right|\ge 1\right\},
$$
we easily have that the second integral in~\eqref{sum.int} is bounded by
$$
\int_{\left\{\sum\limits_{i=0}^{m-1}\tau\circ P^i\leq N+1\right\}}|{(\phi\circ P_n^m)^+-(\phi\circ P^m)^+|d\lambda}+\int_{ A_n}|{(\phi\circ P_n^m)^+-(\phi\circ P^m)^+|d\lambda}.
$$
Observe that 
\begin{equation}\label{eq.and}
\lambda( A_n)\to0 ,\quad\text{as $n\to\infty$},
\end{equation}
because for large $n$, a point belongs to $A_n$ only if it belongs to some small neighborhood of the (finite) set of discontinuity lines of $P^m$.
As $\phi$ is bounded, we have that the second term in the last integral above is bounded by
$2\lambda(A_n)\sup\phi$. Then, \eqref{eq.and}  implies that the second term in that integral is small for sufficiently large $n$. 

It remains to control
$$
\int_{\left\{\sum_{i=0}^{m-1}\tau\circ P^i\leq N+1\right\}}|{(\phi\circ P_n^m)^+-(\phi\circ P^m)^+|d\lambda}.
$$
Observe that the points in $\left\{\sum_{i=0}^{m-1}\tau\circ P^i\leq N+1\right\}$ must necessarily be out of a neighborhood of the discontinuity lines of the  map $P^m$. If $n$ is sufficiently large, then the same holds for $P_n^m$. This means that the return time associated to these maps is uniformly bounded for large $n$. Then, just by the continuous variation of trajectories in finite periods of time, we can make $|(\phi\circ P_n^m)^+-(\phi\circ P^m)^+|$ small for large $n$.
\end{proof}

\begin{lemma}\label{le.limi}
For any $m\ge 1$ we have
$$
{\displaystyle\lim\limits_{n\to\infty}{\int(\phi\circ P_n^m)^+d\bar{\mu}_n}=\int{(\phi\circ P^m)^+d\bar{\mu}}}.
$$
\end{lemma}
\begin{proof}
Given $m\in\mathbb{N}$, then
\begin{align*}
&\left|\int{(\phi\circ P_n^m)^+d\bar{\mu}_n}-\int{(\phi\circ P^m)^+d\bar{\mu}}\right|\leq\\
&\hspace{2cm}\left|\int{(\phi\circ P_n^m)^+d\bar{\mu}_n}-\int{(\phi\circ P^m)^+d\bar{\mu}_n}\right|\\
&\hspace{3cm}+\left|\int{(\phi\circ P^m)^+d\bar{\mu}_n}-\int{(\phi\circ P^m)^+d\bar{\mu}}\right|.
\end{align*}
Using that the density of $\bar\mu_n$ converges to the density of $\bar\mu$ in the $L^1$-norm, by Proposition~\ref{pr.keller} and the fact that $\phi$ is bounded, we easily see that
the second term in the sum above tends to zero when $n$ goes to $\infty$. 
So, we are left to prove that the first term converges to zero when $n$ goes to infinity. In fact, using the uniform boundedness of the densities in Proposition~\ref{pr.keller}, we obtain
\begin{align*}
\left|\int{(\phi\circ P_n^m)^+d\bar{\mu}_n}-\int{(\phi\circ P^m)^+d\bar{\mu}_n}\right|
&\leq\int\left|(\phi\circ P_n^m)^+-(\phi\circ P^m)^+\right|\left|\frac{d\bar{\mu}_n}{d\lambda}\right|d\lambda\\
&\leq M\int|(\phi\circ P_n^m)^+-(\phi\circ P^m)^+|d\lambda
\end{align*}
which, by Lemma~\ref{le.contleb},  can be made arbitrarily small for $n$ sufficiently large.
\end{proof}

\begin{proposition}\label{pro1}
$
\displaystyle\lim\limits_{n\to\infty}\int\phi~d\tilde{\mu}_n=\int\phi~d\tilde{\mu}.
$
\end{proposition}
\begin{proof}
The compactness of $\Sigma$ implies that $\phi$ is uniformly continuous and, therefore, given $\epsilon>0$ there exists $\delta>0$ such that 
\begin{equation}\label{unifcont}
|\phi(\xi_1)-\phi(\xi_2)|<\epsilon,\quad\mbox{for all $\xi_1,\xi_2\in\Sigma$ with $\dist(\xi_1,\xi_2)<\delta$}.
\end{equation}
As we know, the rate of the contraction of the stable foliation on $\Sigma$ is uniform for all vector fields in $\mathcal{X}$. 
So, the first return maps are uniformly contractive. In particular, given $\delta>0$ there exists $m_0>0$ such that for all  $n$ we have 
\begin{equation}\label{smalldiam}
\mbox{diam}(P_n^m)(\gamma)\leq\delta,\quad\mbox{for all }\gamma\in\mathcal F_n\mbox{ and }m\geq m_0.
\end{equation}
Take arbitrary numbers $m_1,m_2$ with  $m_2\geq m_1\geq m_0$.
Given $x\in I$, let $\gamma$ be the leaf in $\mathcal F_n$ containing $x$ and 
$\gamma_{m_2-m_1}$ be the leaf in $\mathcal F_n$ containing $P_n^{m_2-m_1}(\gamma)$. We have
 $$(\phi\circ P_n^{m_2})^+(x)=\sup \phi|_{P_n^{m_2}(\gamma)}=\sup \phi|_{P_n^{m_1}(P_n^{m_2-m_1}(\gamma))}.$$
As $f^{m_2-m_1}(x)\in \gamma_{m_2-m_1}$, we also have
 $$(\phi\circ P_n^{m_1})^+(f^{m_2-m_1}(x))=\sup \phi|_{P_n^{m_1}(\gamma_{m_2-m_1})}$$
Then, since $\gamma_{m_2-m_1}$ contains $P_n^{m_2-m_1}(\gamma)$, it follows from \eqref{unifcont} and \eqref{smalldiam} that
\begin{align*}
\left|(\phi\circ P_n^{m_2})^+(x) -(\phi\circ P_n^{m_1})^+(f^{m_2-m_1}(x))\right|<\epsilon.
\end{align*}
Knowing that $\displaystyle{\int(\phi\circ P_n^{m_1})^+d\bar{\mu}_n=\int(\phi\circ P_n^{m_1})^+\circ f_n^{m_2-m_1}d\bar{\mu}_n}$, because $\bar{\mu}_n$ is an $f_n$ invariant probability measure, we obtain
\begin{align*}
\left|\int(\phi\circ P_n^{m_2})^+d\bar{\mu}_n-\int(\phi\circ P_n^{m_1})^+d\bar{\mu}_n\right|\leq\epsilon.
\end{align*}
Consequently, the sequence $\displaystyle{\left(\int{(\phi\circ P_n^m)^+}d\bar{\mu}_n\right)_{m,n}}$ is uniformly Cauchy, because $m_0$ does not depend on $n$.  Hence,
$$\lim\limits_{n\to\infty}\lim\limits_{m\to\infty}\int{(\phi\circ P_n^m)^+~d\bar{\mu}_n}=\lim\limits_{m\to\infty}\lim\limits_{n\to\infty}\int{(\phi\circ P_n^m)^+~d\bar{\mu}_n}.$$
Therefore,
\begin{align*}
\lim\limits_{n\to\infty}\int\phi~d\tilde{\mu}_n=\lim\limits_{n\to\infty}\lim\limits_{m\to\infty}\int{(\phi\circ P_n^m)^+~d\bar{\mu}_n}
=\lim\limits_{m\to\infty}\lim\limits_{n\to\infty}\int{(\phi\circ P_n^m)^+~d\bar{\mu}_n},
\end{align*}
and by Lemma~\ref{le.limi},
 $${\displaystyle\lim\limits_{n\to\infty}{\int(\phi\circ P_n^m)^+d\bar{\mu}_n}=\int{(\phi\circ P^m)^+d\bar{\mu}}}.$$
Taking limit with $m\to\infty$ we complete the proof, by definition of $\tilde\mu$.

\end{proof}

\section{Statistical stability for the flow}

Now we prove Theorem \ref{main}. Let $(X_n)_{n\geq 1}$ be a sequence in $\mathcal{X}$ converging to $ X\in \mathcal X$ in the $C^2$ topology. Using again shortened subindex notation as in Section~\ref{se.poincare}, we need to  prove that $\mu_n\to\mu$ in the weak$^\ast$ topology. 
Let $\varphi:\bar{U}\to\mathbb{R}$ be any continuous function. We have
$$
\int\varphi~d\mu_n=\frac{1}{\tilde{\mu}_n(\tau_n)}\int\int_0^{\tau_n(\xi)}\varphi(X_n(\xi,t))dt d\tilde{\mu}_n(\xi).
$$
Adding and subtracting the term 
$$
\frac{1}{\tilde{\mu}_n(\tau_n)}\int\int_0^{\tau(\xi)}\varphi(X(\xi,t))dtd\tilde{\mu}(\xi),
$$
we have  $\left|\int\varphi~d\mu_n-\int\varphi~d\mu\right|$ bounded by the sum of two terms
\begin{equation}\label{333}
\left|\frac{1}{\tilde{\mu}_n(\tau_n)}-\frac{1}{\tilde{\mu}(\tau)}\right|\int\int_0^{\tau(\xi)}|\varphi(X(\xi,t))|dtd\tilde{\mu}(\xi).
\end{equation}
and
\begin{equation}\label{444}
\frac{1}{\tilde{\mu}_n(\tau_n)}\left|\int{\int_0^{\tau_n(\xi)}{\varphi(X_n(\xi,t))dtd\tilde{\mu}_n(\xi)-\int\int_0^{\tau(\xi)}\varphi(X(\xi,t))dt}d\tilde{\mu}(\xi)}\right|.
\end{equation}
Our goal now is to show that  the terms in~\eqref{333} and~\eqref{444} converge to zero when $n\to\infty$. 

\begin{lemma}\label{lem01}
$\displaystyle{\lim_{n\to\infty}\int\tau_n~d\tilde{\mu}_n=\int\tau~d\tilde{\mu}}$.
\end{lemma}


\begin{proof}
Let  $\tau_N$ and $\tau_{n,N}$ be defined by
$$\tau_N(\xi)=\min\{\tau(\xi), N\}\quad\text{and}\quad\tau_{n,N}(\xi)=\min\{\tau_n(\xi), N\}.
$$
Observe that for each fixed $ N\ge 1$ we have that $\tau_{N}$ and $\tau_{n,N}$  are bounded continuous functions, and $\tau_{n,N}$ converges uniformly  to $\tau_{N}$, as $n\to\infty$.
We have
\begin{align}\label{222}
\left|\int\right. &\tau_n~d\tilde{\mu}_n \left.-\int\tau~d\tilde{\mu}\right|\le \nonumber \\
& \left |\int\tau_n~d\tilde{\mu}_n-\int\tau_{n,N}~d\tilde{\mu}_n\right |+\left |\int\tau_{n,N}~ d\tilde{\mu}_n-\int\tau_N ~d\tilde{\mu}\right |+\left |\int\tau_N ~d\tilde{\mu}-\int\tau ~d\tilde{\mu}\right |.
\end{align}
Now we prove that the first term  is  small for large  $N$ (uniformly in $n$), and the calculation is similar for the third term.
We  have 
\begin{equation}\label{eq:tausatr}
\left |\int\tau_n~d\tilde{\mu}_n-\int\tau_{n,N}~d\tilde{\mu}_n\right |=\int(\tau_n-\tau_{n,N})~d\tilde{\mu}_n\leq\int(\tau_n-\tau_{n,N})^+~d\tilde{\mu}_n.
\end{equation}
Now, for $n,N,m\ge 1$ let 
$(\tau_n-\tau_{n,N})_{m}^+$
be defined by
$$(\tau_n-\tau_{n,N})_{m}^+(\xi)=\min\left\{(\tau_n-\tau_{n,N})^+(\xi),m\right\}.$$
Note that $(\tau_n-\tau_{n,N})_{m}^+$ converges monotonically to  $(\tau_n-\tau_{n,N})^+$  as $m\to\infty$. Note also that the functions $(\tau_n-\tau_{n,N})_{m}^+$ and $(\tau_n-\tau_{n,N})^+$ can be interpreted as the functions defined  in~$I$   introduced in~\eqref{EQ.+-}. Using the Monotone Convergence Theorem, the definition of $\tilde\mu_n$ and the uniform boundedness on the density of $\bar\mu_n$ we can write
\begin{align*}\int(\tau_n-\tau_{n,N})^+~d\tilde{\mu}_n
&=\lim\limits_{m\to\infty}\int(\tau_n-\tau_{n,N})_{m}^+~d\tilde{\mu}_n\\
&=\lim\limits_{m\to\infty}\int(\tau_n-\tau_{n,N})_{m}^+~d\bar{\mu}_n
\\
&\leq M\lim\limits_{m\to\infty}\int(\tau_n-\tau_{n,N})_{m}^+~d\lambda\\
&= M\int(\tau_n-\tau_{n,N})^+~d\lambda.
\end{align*}
Now, defining $A_{n,N}=\{x\in I :-C\log(x-O_n)>N\}$ by \eqref{eq.retime} we have 
$$\int(\tau_n-\tau_{n,N})^+~d\lambda\le \int_{A_{n,N}}-C\log(x-O_n)~dx=\int_0^{e^{-C/N}}-C\log x \,dx,$$
and this last integral is clearly  small  for large  $N$ (uniformly in $n$).

Finally, for each $N\ge 1$, the second term in~\eqref{222}  converges to zero as $n$ goes to $\infty$, since $\tilde\mu_n$ converges weakly to $\tilde\mu$, by Proposition~\ref{pro1},  and the functions $\tau_{n,N}$ are  continuous and converge uniformly to $\tau_N$ when $n\to\infty$.
\end{proof}

 Lemma \ref{lem01} implies that (\ref{333}) converges to zero as $n$ goes to infinity, since
$$
\left|\frac{1}{\tilde{\mu}_n(\tau_n)}-\frac{1}{\tilde{\mu}(\tau)}\right|\int\int_0^{\tau(\xi)}|\varphi(X(\xi,t))|dtd\tilde{\mu}(\xi)\leq \left|\frac{1}{\tilde{\mu}_n(\tau_n)}-\frac{1}{\tilde{\mu}(\tau)}\right|\|\varphi\|_{\infty}\tilde{\mu}(\tau).
$$ 
The next result  implies that (\ref{444}) converges to zero when $n$ goes to infinity.
\begin{lemma}
$ \displaystyle\lim\limits_{n\to+\infty}\int\int_0^{\tau_n(\xi)}\varphi(X_n(\xi,t))dtd\tilde{\mu}_n(\xi)=\int\int_0^{\tau(\xi)}\varphi(X(\xi,t))dtd\tilde{\mu}(\xi).
$
\end{lemma}
\begin{proof}
We define
$$
h(\xi)=\int_0^{\tau(\xi)}\varphi(X(\xi,t))~dt,\quad h_n(\xi)=\int_0^{\tau_n(\xi)}\varphi(X_n(\xi,t))~dt,
$$
and using the same notations of the proof of~Lemma~\ref{lem01}, we also define for  $N\ge 1$
$$
h_N(\xi)=\int_0^{\tau_N(\xi)}\varphi(X(\xi,t))~dt,\quad 
h_{n,N}(\xi)=\int_0^{\tau_{n,N}(\xi)}\varphi(X_n(\xi,t))~dt.
$$
The difference,
$$
\displaystyle\left|\int\int_0^{\tau_n(\xi)}\varphi(X_n(\xi,t))dtd\tilde{\mu}_n(\xi)-\int\int_0^{\tau(\xi)}\varphi(X(\xi,t))dtd\tilde{\mu}(\xi)\right|,
$$
is bounded by
$$
\left|\int h_nd\tilde{\mu}_n-\int h_{n,N}d\tilde{\mu}_n\right|+\left|\int h_{n,N} d\tilde{\mu}_n-\int h_N d\tilde{\mu}\right|+\left|\int h_N d\tilde{\mu}-\int h d \tilde{\mu}\right|.
$$
The first  term is bounded by 
$$\|\varphi\|_\infty\int(\tau_n-\tau_{n,N})d\tilde{\mu}_n\leq \|\varphi\|_\infty\int(\tau_n-\tau_{n,N})^+d\tilde{\mu}_n.$$ 
As we saw in the proof of Lemma \ref{lem01}, this last term is small for large enough $N$ (uniformly in $n$), and a similar conclusion holds for the third term. It is easily seen that $h_{n,N}$'s are continuous and poitwise convergence to $h_N$. The second term also follows as in the proof of Lemma \ref{lem01}, because the functions $h_{n,N}$ are continuous and converge uniformly to $h_N$ when $n\to\infty$.
\end{proof}

\end{document}